\documentclass[11pt]{article}

\usepackage{amssymb}
\usepackage{amsthm}
\usepackage{amsmath}
\usepackage{latexsym}
\usepackage{amsfonts}
\usepackage{color}

\DeclareSymbolFont{rsfscript}{OMS}{rsfs}{m}{n}
\DeclareSymbolFontAlphabet{\mathrsfs}{rsfscript}
\DeclareMathOperator{\Fix}{Fix}
\DeclareMathOperator{\Rad}{Rad}

\DeclareMathOperator{\MT}{MT}

\DeclareMathOperator{\End}{End}

\DeclareMathOperator{\Curl}{Curl}
\DeclareMathOperator{\Per}{Per}

\DeclareMathOperator{\St}{St}
\DeclareSymbolFont{rsfscript}{OMS}{rsfs}{m}{n}
\newtheorem{theorem}{Theorem}[section]
\newtheorem{prop}[theorem]{Proposition}

\newtheorem{lemma}[theorem]{Lemma}
\newtheorem{cor}[theorem]{Corollary}

\newtheorem{prob}[theorem]{Problem}

\newcommand{\wt}{\widetilde}
\newcommand{\wh}{\widehat}
\newcommand{\oo}{\overline}
\newcommand{\inv}{^{-1}}

\def\mapright#1{\smash{\mathop{\longrightarrow}\limits^{#1}}}

\title{On periodic points of free inverse monoid endomorphisms}
\author{{\bf Emanuele Rodaro, Pedro V. Silva}\\ $ $\\ {\em Centro de
Matem\'{a}tica, Faculdade de Ci\^{e}ncias, Universidade do
Porto,}\\ {\em R. Campo Alegre 687, 4169-007 Porto, Portugal}\\
{\em e-mail:} emanuele.rodaro@fc.up.pt, pvsilva@fc.up.pt}
\date{\today}

\begin{document}
\maketitle

\begin{abstract}
\noindent It is proved that the periodic point submonoid of a free
inverse monoid endomorphism is always finitely generated. Using 
Chomsky's hierarchy of languages, we prove that the fixed point
submonoid of an endomorphism of a free inverse monoid can be
represented by a context-sensitive language but, in general, it cannot
be represented by a context-free language. 
\end{abstract}

\section{Introduction}
The dynamical study of the automorphisms of a free group and their
space of ends is a 
well established subject in discrete Dynamical Systems. Hence it is a
natural issue to explore these problems in the more general setting of
semigroups. In \cite{CS2}, Cassaigne and the second author studied
finiteness conditions for the infinite fixed 
points of (uniformly continuous) endomorphisms of monoids defined by
special confluent 
rewriting systems, extending results known for free monoids
\cite{PP}. This line of research was pursued by the second author in
subsequent papers \cite{Sil1,Sil2,Sil3}. Recently, in a joint paper
with Sykiotis \cite{RSS}, we considered the case of graph groups,
studying the periodic 
and fixed points of endomorphisms. Inspired by this work, we studied
in \cite{RS} the case of endomorphisms of trace monoids
and their extensions to real traces. 

In this paper we
consider the case of endomorphisms of free inverse monoids. Inverse
monoids are a natural generalization of groups. Indeed, while groups
are represented by bijective functions, inverse semigroups are
represented by partial injective functions. Although these structures
seem to be close, the behavior of their respective free objects can be
radically different. In this 
paper we provide further evidence to this viewpoint. In fact, we show
that 
the inverse submonoid of the points fixed by an endomorphism of a free
inverse monoid is not in general finitely generated as in the group
case. However, we can still frame the fixed point submonoid in 
Chomsky's hierarchy, at the context-sensitive level. A better outcome
is achieved for the periodic point submonoid (finitely generated,
hence rational) and for some related
submonoids called radicals (context-free). 

The paper is organized as follows. Section \ref{sec: preliminaries} is
devoted to preliminaries. In Section \ref{sec: fix and periodic} we
show that the inverse submonoid of periodic points is finitely
generated, while the inverse submonoid of fixed points is not in
general context-free. This is achieved by proving a gap theorem for
endomorphisms whose induced endomorphisms in the free group has no
nontrivial fixed points in the boundary. In Section \ref{sec:radical},
we introduce a family of particular inverse submonoids called
radicals, useful in 
describing the fixed points submonoid. We show that the radical behaves
better than the fixed points submonoids by proving that the $N$th radical is
context-free for suitable $N$. Using this result we finally frame the
fixed point
submonoid in the context-sensitive class. Finally, we raise some open
problems in the last section.

\section{Preliminaries}\label{sec: preliminaries}

The reader is assumed to have some familiarity with the basics of
formal language theory. For references, see \cite{Ber, HMU}.

\subsection{Free groups}

Let $A$ be a finite alphabet and let $\tilde{A}=A\cup A\inv$ be the
\emph{involutive alphabet} where $A\inv$ is the set of \emph{formal}
inverses of $A$. The operator $\inv:A\rightarrow A\inv: a\mapsto
a\inv$ is extended to an involution on the free monoid $\wt{A}^*$ through
$$
1\inv = 1, \;\; (a\inv)\inv=a, \;\; (uv)\inv=v\inv u\inv\;\;\; (a\in
A;\;u,v\in\wt{A}^*).
$$
Let $\sim$ be the congruence on $\wt{A}^*$ generated by the relation
$\{(aa\inv,1)\mid a \in \wt{A} \}$. The
quotient $F_A= \wt{A}^*/\sim$ is the \emph{free group} on $A$, and
$\sigma:\wt{A}^* \to F_A$ denotes the canonical homomorphism.

We denote by $R_A$ the set of all reduced words on
$\wt{A}^*$, i.e.
$$
R_A=\tilde{A}^*\setminus\bigcup_{a\in\tilde{A}}\tilde{A}^*aa\inv\tilde{A}^*
$$
For each $u\in\tilde{A}^*$, $\overline{u}\in R_A$ is the (unique)
reduced word $\sim$-equivalent to $u$. We write also
$\overline{u\sigma}=\overline{u}$. As usual, we often
identify the elements of $F_A$ with their reduced representatives. For $B\subseteq \wt{A}^*$, $\oo{B}$ denotes the set of reduced words of $B$.

We describe now the {\em boundary} of $F_A$. Given $u,v \in F_A$,
written in reduced form, we denote by $u\wedge v$ the longest common
prefix of $u$ and $v$. Let
$$r(u,v) = \left\{
\begin{array}{ll}
|u\wedge v|&\mbox{ if }u \neq v\\
+\infty&\mbox{ otherwise}
\end{array}
\right.$$
and define $d(u,v) = 2^{-r(u,v)}$ (using the convention $2^{-\infty} =
0$). Then $d$ is know as the {\em prefix metric} on $F_A$. In fact, $d$
is an ultrametric since $$d(u,w) \leq \max\{ d(u,v),d(v,w) \}$$ holds
for all $u,v,w \in F_A$.

The metric space $(F_A,d)$ is not complete, but its completion admits
a simple description: we add to $F_A$ all the (right) infinite reduced
words $a_1a_2a_3\ldots$ on $\wt{A}$. These new elements are called the
{\em boundary} of $F_A$ and the completion (which is indeed compact)
is denoted by
$\wh{F_A}$. The metric on $\wh{F_A}$ admits the same description as
$d$ and is also denoted by $d$.

By standard topology results  (see \cite[Section XIV.6]{Dug}), any
uniformly continuous endomorphism $\theta$ of $F_A$ admits a (unique)
continuous extension $\wh{\theta}: \wh{F_A} \to \wh{F_A}$. Thus, if
$(u_n)_n$ is a sequence on $F_A$ converging in $\wh{F_A}$, then
$$(\lim_{n\to \infty}u_n)\wh{\theta} = \lim_{n\to \infty}
u_n\theta$$
by continuity. The uniformly continuous endomorphisms of $F_A$ turn
out to be the monomorphisms. For details on the boundary and
endomorphism extensions on a general context (containing free groups
as a particular case), the reader is referred to \cite{CS}.

\subsection{Free inverse monoids}

A monoid $M$ is said to be {\em inverse} if
$$\forall u \in M\; \exists ! u\inv \in M: (uu\inv u = u \: \wedge \; u\inv
uu\inv = u\inv).$$
For details on inverse monoids, see \cite{Pet}.

Let $E(M)$ denote the subset of idempotents of $M$. Then $M$ is
inverse if and only if $$\forall u \in M\; \exists  v \in M: uvu = u$$
and $ef = fe$ for all $e,f \in E(M)$. Furthermore,
by the well-known Vagner-Preston representation theorem, inverse
monoids are represented by partially injective transformations as
groups are represented by permutations.

We recall also that there is a multiplicative partial order (the {\em natural
  partial order}) defined on $M$ by
$$s\le t \quad \mbox{if } s=et  \mbox{ for some }e\in E(M).$$

Let $\rho$ be the congruence on $\wt{A}^*$ generated by the relation
$$\{(uu\inv u,u)\mid u \in \wt{A}^* \} \cup \{(uu\inv vv\inv,vv\inv
uu\inv )\mid u,v \in \wt{A}^* \},$$
known as the {\em Vagner congruence} on $\wt{A}^*$. The
quotient $M_A = \wt{A}^*/\rho$ is the \emph{free inverse monoid} on $A$, and
$\pi:\wt{A}^* \to M_A$ denotes the canonical homomorphism. Note that
$\pi$ is {\em matched}, since $u\inv\pi = (u\pi)\inv$ for every $u \in
\wt{A}^*$. Moreover, there exists a canonical homomorphism
$\sigma':M_A \to F_A$ such that $\sigma = \pi\sigma'$. In fact,
$(u\rho)\sigma' = \oo{u}$ for every $u \in \wt{A}^*$.

Two normal forms for $M_A$ were introduced independently by Scheiblich
\cite{Sch} and Munn \cite{Mun}, following respectively an algebraic
and a geometric approach. We can make them converge with the help of
the {\em Cayley graph} $C_A$ of $F_A$. Indeed, let $C_A$ have vertex set
$F_A$ and edges $g \mapright{a} ga$ for all $g \in F_A$ and $a \in
\wt{A}$. Note that $C_A$ is an {\em inverse graph}:
$(p,a,q)$ is an edge of $C_A$, so is $(q,a\inv,p)$. Such edges are
said to be the inverse of each other.
It is well known that, if $A \neq \emptyset$, the Cayley
graph $C_A$ is an infinite tree (if we view pairs of inverse edges as a single
undirected edge).

We can identify $M_A$ with the set $M'_A$ of all ordered pairs
$(\Gamma,g)$, where $\Gamma$ is a finite connected inverse subgraph of $C_A$
having both $1$ and $g$ as vertices. If we view $\Gamma$ as a finite
birooted tree (for roots 1 and $g$), we have the Munn
representation. If we focus on its set
of vertices (a finite prefix closed subset of $F_A$), we have the
Scheiblich representation. Now $F_A$ acts on the left of $C_A$ in the
obvious way and the multiplication on $M'_A$ can be given through
$$(\Gamma,g)\cdot(\Gamma',g')=(\Gamma\cup g\Gamma',gg'),$$
and inversion through
$$(\Gamma,g)\inv = (g\inv\Gamma, g\inv).$$
Finally, the homomorphism $\pi$ translates into $\pi':\wt{A}^* \to M'_A$
given by $u\pi = (\MT(u),u\sigma)$, where $\MT(u)$ (the {\em Munn
  tree} of $u$) is the finite inverse subgraph of $C_A$ defined by
reading the path $1 \mapright{u} g$ in $C_A$. We shall usually identify
$u\sigma$ with $\oo{u}$ and $\MT(u)$ with the prefix-closed set $T(u)$
of $R_A$ having as elements the reduced forms of the vertices of $\MT(u)$.

Note that
$$T(uu\inv) = T(u)$$
holds for every $u \in \wt{A}^*$. Moreover, for all $u,v \in
\wt{A}^*$, we have
$$u\rho \leq v\rho \hspace{.7cm} \Leftrightarrow \hspace{.7cm} T(u)
\supseteq T(v) \; \wedge \; \oo{u} = \oo{v}.$$
Also, if we view $\MT(u)$ as an automaton $\mathcal{A}_u$ with
initial vertex 1 and terminal vertex $u$, then
$$L(\mathcal{A}_u) = \{ v \in \wt{A}^* \mid v \geq u\mbox{ in }M_A \}.$$
Such basic properties of $M_A$ will be used throughout the paper
without further reference.

Finally, we define the norm $||u|| = \max\{ |\oo{v}|: v \in T(u)\}$
for every $u \in M_A$.

\subsection{Chomsky's hierarchy in $M_A$}

Following \cite{Ber}, we can define {\em rational subsets} of $M_A$
with the help of the homomorphism $\pi$: we say that $X \subseteq M_A$
is rational if $X = L\pi$ for some rational $L \subseteq \wt{A}^*$. This idea
can be extended to higher classes of languages
in the Chomsky's hierarchy, for instance we say that $X \subseteq M_A$ is
\emph{context-free} (respectively \emph{context-sensitive}) whenever
$X = L\pi$ for some context-free (respectively context-sensitive) $L
\subseteq \wt{A}^*$.

Actually, in the case of rational subsets, the concept is independent
from the homomorphism considered, and this property holds for
arbitrary monoids.

Anisimov and Seifert's Theorem (see \cite{Ber}) states that a subgroup
of a group $G$
is rational if and only if it is finitely generated. We note that there is an
analogous of this theorem for inverse semigroups \cite[Lemma
3.6]{FreeInvGraphImm} stating that a closed inverse subsemigroup
of an inverse semigroup is rational if and only if it is finitely
generated (here closed meaning upper set for the natural partial order).

\section{Fixed points and periodic points}\label{sec: fix and periodic}

Given a monoid $M$, we denote by $\End(M)$ the monoid of endomorphisms
of $M$. Given $\varphi \in \End(M)$,
$$\Fix(\varphi)=\{x\in M: x\varphi = x\}$$
is the submonoid of {\em fixed points} of $\varphi$, and
$$\Per(\varphi) = \cup_{n \geq 1} \Fix(\varphi^n)$$
is the submonoid of {\em periodic points} of $\varphi$ (note that
$\Fix(\varphi^m) \subseteq \Fix(\varphi^n)$ whenever $m|n$, hence
$(\Fix(\varphi^k))(\Fix(\varphi^{\ell})) \subseteq
\Fix(\varphi^{k\ell})$ and so $\Per(\varphi)$ is indeed a
submonoid). Moreover, if $M$ is an inverse monoid (a group), then both
$\Fix(\varphi)$ and $\Per(\varphi)$ are inverse submonoids (subgroups).

Let
$\theta \in \End(F_A)$. By \cite{GT2}, $\Fix(\theta)$ is finitely
generated. Now, by Nielsen's Theorem (see \cite{LS}), every
subgroup of a free group is free. Recall that the {\em rank} of a free
group $F$ is the cardinality of a basis of $F$.
By \cite{IT}, every subgroup $\Fix(\theta^n)$ has actually rank $\le |A|$.
According to Takahasi's Theorem \cite{Tak}, every ascending chain of
subgroups of bounded rank of $F_A$ must be stationary, and so must be
$$\Fix(\theta^{1!}) \subseteq \Fix(\theta^{2!}) \subseteq
\Fix(\theta^{3!}) \subseteq \ldots$$
Since $\Per(\theta) = \cup_{n \geq 1} \Fix(\theta^{n!})$, this
provides a proof for the following well-known result:

\begin{prop}
\label{perfix}
Let $\theta \in \End(F_A)$. Then $\Per(\theta) = \Fix(\theta^N)$ for
some $N \geq 1$.
\end{prop}

The minimum integer $N$ such that $\Per(\theta)=\Fix(\theta^N)$ is
called the \emph{curl} of $\theta$ and is denoted by $\Curl(\theta)$.

\begin{lemma}
\label{redcurl}
Let $\varphi \in \End(M_A)$ and let $N = \Curl(\oo{\varphi})$. Let $k
\ge 1$. Then $\Curl(\oo{\varphi}^{Nk}) = 1$.
\end{lemma}

\begin{proof}
We have
$$\Fix(\oo{\varphi}^{Nk}) \subseteq \Per(\oo{\varphi}^{Nk}) \subseteq
\Per(\oo{\varphi}) = \Fix(\oo{\varphi}^N) \subseteq \Fix(\oo{\varphi}^{Nk}),$$
hence $\Fix(\oo{\varphi}^{Nk}) = \Per(\oo{\varphi}^{Nk})$ and so
$\Curl(\oo{\varphi}^{Nk}) = 1$.
\end{proof}

Assume now that $\varphi \in \End(M_A)$. Then $\varphi$ induces an endomorphism
$\overline{\varphi}\in\End(F_A)$ defined by $\overline{\varphi} =
(\sigma')\inv\varphi\sigma'$.
It is straightforward to check that $\overline{\varphi}$ is a
well defined homomorphism using the fact that $\varphi$ sends
idempotents into idempotents.

\begin{lemma}
\label{permv}
Let $\varphi \in \End(M_A)$ and let $u \in M_A$. Then the
following conditions are equivalent:
\begin{itemize}
\item[(i)]
$u \in \Fix(\varphi)$;
\item[(ii)]
$uu\inv \in \Fix(\varphi)$ and $\oo{u} \in \Fix(\oo{\varphi})$.
\end{itemize}
Moreover, in this case we have $(T(u))\oo{\varphi} \subseteq T(u)$.
\end{lemma}

\begin{proof}
(i) $\Rightarrow$ (ii). The first part follows from $\Fix(\varphi)$ being an
inverse submonoid of $M_A$. On the other hand, we have
$$\oo{u}\,\oo{\varphi} = \oo{u}(\sigma')\inv\varphi\sigma' =
u\varphi\sigma' = u\sigma' = \oo{u}.$$

(ii) $\Rightarrow$ (i). We have $\oo{u}\,\oo{u}\inv \geq uu\inv$,
hence  $(\oo{u}\,\oo{u}\inv)\varphi \geq (uu\inv)\varphi =
uu\inv$. It follows that $\oo{u}\varphi = e\oo{u}$ for
some idempotent $e \geq uu\inv$ and so
$$u\varphi = (uu\inv \oo{u})\varphi = uu\inv e\oo{u} =
uu\inv \oo{u} = u.$$

Finally, let $v \in T(u)$. Then $vv\inv
\geq uu\inv$ and $uu\inv \in \Fix(\varphi)$ yields
$(vv\inv)\varphi \geq (uu\inv)\varphi =
uu\inv$. Hence $T(v\varphi) = T((vv\inv)\varphi)
\subseteq T(uu\inv) = T(u)$. Thus
$$v\oo{\varphi} =
v(\sigma')\inv\varphi\sigma' = v\varphi\sigma' \in
T(v\varphi) \subseteq T(u)$$
as required.
\end{proof}

Let $\varphi \in \End(M_A)$. An idempotent $e\in E(M_A)$ is said to be
$\varphi$-\emph{stable} (or
simply stable when the endomorphism $\varphi$ is clear from the
context) whenever the {\em orbit} $\{ e\varphi^n \mid n \geq 0\}$ is
finite. It follows that  $\{ e\varphi^n \mid n \geq 0\} = \{ e,
e\varphi, \ldots, e\varphi^{p-1} \}$ for some $p \geq 1$. We call the smallest
such $p$ the \emph{period} of $e$ with respect to $\varphi$.
By minimality of $p$, the idempotents
$e, e\varphi, e\varphi^{2},\ldots, e\varphi^{p-1}$ are all
distinct.
Finally, we define
$$
  \mathcal{K}_{\varphi}(e)=\prod_{i=0}^{p-1} e\varphi^{i},
$$
$$\St_{\varphi} = \{ a \in \wt{A} \mid aa\inv \mbox{ is
  $\varphi$-stable} \}.
$$
We show that $\St_{\varphi} \neq \emptyset$ if
there exist nontrivial fixed points, but first we prove a simple
lemma. Given $u,v \in \wt{A}^*$ and a vertex $p$ of $\MT(u)$, we say
that $\MT(v)$ embeds in $\MT(u)$ at $p$ if $\oo{pT(v)} \subseteq
T(u)$.

\begin{lemma}\label{lem: action on edges}
  Let $\varphi\in \End(M_A)$ and $u \in M_A$. Let $p
  \mapright{v} q$ be a path in $\MT(u)$. Then
  $\MT(v\varphi)$ embeds in $\MT(u\varphi)$ at
  $p\oo{\varphi}$.
\end{lemma}

\begin{proof}
Considering $p$ in reduced form, there exists a path $1 \mapright{p} p
\mapright{v} q \mapright{w} \oo{u}$ in $\MT(u)$ and so $pvw \geq
u$ in $M_A$. Hence $(pvw)\varphi \geq u\varphi$ and so
$T((pvw)\varphi) \subseteq T(u\varphi)$. It follows that
$$\oo{(p\oo{\varphi})T(v\varphi)} = \oo{(p\varphi)
  T(v\varphi)} \subseteq
T((p\varphi)(v\varphi)(w\varphi)) =
T((pvw)\varphi) \subseteq T(u\varphi)$$
as claimed.
\end{proof}

\begin{prop}
\label{prop: edge are stable}
Let $\varphi\in\End(M_A)$ and let $u \in \Fix(M_A)$. If $a \in
\wt{A}$ labels some edge in $\MT(u)$, then $a \in \St_{\varphi}$.
\end{prop}

\begin{proof}
 Let $p\mapright{a}q$ be an edge of $\MT(u)$. By Lemma \ref{lem:
   action on edges},
 $\MT((aa\inv)\varphi^n)$ embeds in $\MT(u\varphi^n) = \MT(u)$ at
 $p\oo{\varphi^n}$ for every $n\ge 0$. This bounds the size of
 $\MT((aa\inv)\varphi^n)$, hence the orbit of $aa\inv$ is finite and
 so $a \in \St_{\varphi}$.
\end{proof}

We prove yet another simple lemma:

\begin{lemma}
\label{stablepower}
Let $\varphi\in\End(M_A)$ and let $n \ge 1$. Then
$\St_{\varphi^n} = \St_{\varphi}$.
\end{lemma}

\begin{proof}
Let $a \in \St_{\varphi^n}$. Then $(aa\inv)\varphi^{nk} = (aa\inv)\varphi^{n(k+p)}$
for some $k \geq 0$ and $p \geq 1$. Hence the $\varphi$-orbit of
$aa\inv$ is finite and so $a \in \St_{\varphi}$. Since the $\varphi$-orbit of
$aa\inv$ contains the $\varphi^n$-orbit, the converse implication follows.
\end{proof}

The set of $\varphi$-tiles is defined by
$$\mathcal{T}_{\varphi} = \{ \mathcal{K}_{\varphi}(aa\inv)a \mid a \in
\St_{\varphi} \}.$$

\begin{theorem}\label{theo: product of phi-tiles}
Let $\varphi \in \End(M_A)$ be such that $\Curl(\oo{\varphi}) =
1$. Then $\Fix(\varphi) = \mathcal{T}_{\varphi}^*$.
\end{theorem}

\begin{proof}
Let $a \in St_{\varphi}$ and $t =
\mathcal{K}_{\varphi}(aa\inv)a$. Let $p$ denote the
period of $aa\inv$. Since $aa\inv \geq
\mathcal{K}_{\varphi}(aa\inv)$, we have $tt\inv = \mathcal{K}_{\varphi}(aa\inv)$.
By Lemma \ref{lem: action on edges} applied to $\varphi^n$ we know that
$\MT(a\varphi^n)$ embeds in $\MT((aa\inv)\varphi^n)$ and therefore in
$\MT(t)$. Thus
$a\oo{\varphi^n}$ labels a path in $\MT(t)$ for every $n$. Since
$\MT(t)$ is a finite tree, it admits only finitely many paths of
reduced label, hence the orbit $\{ a\oo{\varphi^n} \mid n \geq 0 \}$
must be finite. Since $\oo{\varphi^n} = \oo{\varphi}^n$, it follows
that $a \in \Per(\oo{\varphi})$ and so $a \in \Fix(\oo{\varphi})$
since $\Curl(\oo{\varphi}) = 1$. Hence $aa\inv \geq (aa\inv)\varphi$
and so
$$(tt\inv)\varphi = (\prod_{i=0}^{p-1} (aa\inv)\varphi^{i})\varphi =
\prod_{i=1}^{p} (aa\inv)\varphi^{i} = \prod_{i=0}^{p-1}
(aa\inv)\varphi^{i} = tt\inv.$$
Since $\oo{t} = a \in \Fix(\oo{\varphi})$, it follows from Lemma
\ref{permv} that $t \in \Fix(\varphi)$. Hence
$\mathcal{T}_{\varphi} \subseteq \Fix(\varphi)$ and so $\mathcal{T}_{\varphi}^*
 \subseteq \Fix(\varphi)$.

Conversely, let $u \in \Fix(\varphi)$. Write $u = a_1 \ldots a_m$
with $a_i \in \wt{A}$. For $i = 1,\ldots,m$, let $t_i =
\mathcal{K}_{\varphi}(a_ia_i\inv)a_i$. By Proposition \ref{prop: edge
  are stable}, we have $a_i \in \St_{\varphi}$ and so $t_i \in
\mathcal{T}_{\varphi}$ for every $i$.

We show that
$u = t_1\ldots t_m$ in $M_A$. Indeed, by Lemma \ref{lem: action on
  edges}, $\MT((a_ia_i\inv)\varphi^n)$ embeds in $\MT(u\varphi^n) =
\MT(u)$ at $(\oo{a_1\ldots a_{i-1}})\oo{\varphi}^n$ for every $n \ge
0$. Hence the orbit $\{ (\oo{a_1\ldots a_{i-1}})\oo{\varphi}^n \mid n
\ge 0 \}$ must be finite and since $\Curl(\oo{\varphi}) = 1$ we get
$\oo{a_1\ldots a_{i-1}} \in \Per(\oo{\varphi}) =
\Fix(\oo{\varphi})$. Thus $(\oo{a_1\ldots a_{i-1}})\oo{\varphi}^n = \oo{a_1\ldots
  a_{i-1}}$ and it follows easily that $\MT(t_i)$ embeds in $\MT(u)$
at $\oo{t_1\ldots t_{i-1}}$ and so $T(t_1\ldots t_m) \subseteq T(u)$.

Now, since  $a_i = \oo{t_i}$, it is not difficult to check
that $t_1...t_m\le u$, from which the converse inclusion follows as well. Together with $\oo{t_1\ldots t_m} =
\oo{a_1\ldots a_m} = \oo{u}$, this implies $u = t_1\ldots t_m$ in
$M_A$. Therefore $\Fix(\varphi) = \mathcal{T}_{\varphi}^*$.
\end{proof}

We can now solve completely the case of periodic points:

\begin{theorem}\label{perio}
Let $\varphi \in \End(M_A)$. Then $\Per(\varphi)$ is finitely generated.
\end{theorem}

\begin{proof}
Let $N = \Curl(\oo{\varphi})$. We show that
\begin{equation}
\label{perio1}
\Per(\varphi) = (\bigcup_{k \geq 1} \mathcal{T}_{\varphi^{Nk}})^*.
\end{equation}
Indeed, let $u \in \Per(\varphi)$. Then $u = u\varphi^k$ for some $k
\ge 1$ and so $u = u\varphi^{Nk}$. By Lemma \ref{redcurl}, we
have $\Curl(\oo{\varphi^{Nk}}) = 1$, hence
$u \in \Fix(\varphi^{Nk}) = \mathcal{T}_{\varphi^{Nk}}^*$ by Theorem
\ref{theo: product of phi-tiles}.

Conversely, we have $\mathcal{T}_{\varphi^{Nk}} \subseteq
\Fix(\varphi^{Nk}) \subseteq \Per(\varphi)$ for every $k$ by Theorem
\ref{theo: product of phi-tiles}. Since $\Per(\varphi)$ is a
submonoid, (\ref{perio1}) holds.

Now it suffices to show that $B = \bigcup_{k \geq 1}
\mathcal{T}_{\varphi^{Nk}}$ is finite. Indeed, if $t \in
\mathcal{T}_{\varphi^{Nk}}$, then $t = \mathcal{K}_{\varphi^{Nk}}(aa\inv)a$
for some $a \in \St_{\varphi^{Nk}}$. Since there are only finitely
many choices for $a \in \wt{A}$ and the $\varphi$-orbit of
$aa\inv$ is finite for every such $a$, it follows that $B$ is finite
and so $\Per(\varphi)$ is finitely generated.
\end{proof}

\begin{cor}
\label{ppff}
Let $\varphi \in \End(M_A)$. The following conditions are
equivalent:
\begin{itemize}
\item[(i)]
$\Fix(\varphi)$ is infinite;
\item[(ii)]
$\Per(\varphi)$ is infinite;
\item[(iii)]
$\Per(\varphi) \not\subseteq E(M_A)$.
\end{itemize}
\end{cor}

\begin{proof}
(i) $\Rightarrow$ (ii). Trivial.

(ii) $\Rightarrow$ (i). We build an infinite sequence $(e_n)_n$ of (distinct)
elements of $\Fix(\varphi)$ as follows. Let $n \geq 1$ and assume that
$e_1, \ldots, e_{n-1}$ are already defined. Since $\Per(\varphi)$ is
infinite, there exists some $u \in \Per(\varphi)$ such that
$||u|| > ||e_i||$ for $i = 1,\ldots, n-1$. We have $uu\inv \in
\Per(\varphi)$, say $(uu\inv)\varphi^m = uu\inv$. Clearly, we
can take
$e_n = \mathcal{K}_{\varphi}(uu\inv)=\prod_{i=0}^{m-1} (uu\inv)\varphi^{i}
\in \Fix(\varphi)$. Since $||e_n|| \geq ||u|| > ||e_i||$, then $e_n
\neq e_i$ for $i = 1,\ldots, n-1$. Thus we build an infinite
sequence $(e_n)_n$ and so $\Fix(\varphi)$ is infinite.

(ii) $\Rightarrow$ (iii). By Theorem \ref{perio}, $\Per(\varphi)$ is
finitely generated, and every finitely generated submonoid of $E(M_A)$
is finite.

(iii) $\Rightarrow$ (ii). If $u \in \Per(\varphi) \setminus E(M_A)$, then
$u^*$ is an infinite submonoid of $\Per(\varphi)$.
\end{proof}

Given $\varphi \in \End(M_A)$ with $\oo{\varphi}$ injective, we may
write $\wh{\varphi} = \wh{\oo{\varphi}}$. Recall that we need
injectivity to extend $\oo{\varphi}$ to the boundary of $F_A$.
If there are no nontrivial fixed points in
$\wh{\varphi}$, a good deal of the hierarchy collapses:

\begin{theorem}
\label{noncf}
Let $\varphi \in \End(M_A)$ be such that $\oo{\varphi}$ is injective
and $\Fix(\wh{\varphi}) = 1$. Then the following conditions are
equivalent:
\begin{itemize}
\item[(i)]
$\Fix(\varphi)$ is context-free;
\item[(ii)]
$\Fix(\varphi)$ is rational;
\item[(iii)]
$\Fix(\varphi)$ is finitely generated;
\item[(iv)]
$\Fix(\varphi)$ is finite;
\item[(v)]
$\Per(\varphi)$ is finite;
\item[(vi)]
$\Per(\varphi) \subseteq E(M_A)$.
\end{itemize}
\end{theorem}

\begin{proof}
(i) $\Rightarrow$ (iv). Assume that $\Fix(\varphi) = K\pi$ for some $K \subseteq
\wt{A}^*$ context-free.
By the Pumping Lemma for context-free languages, there exists a
constant  $p \geq 1$ such that every $w \in K$ of length $> p$ admits a
factorization $w = w_1w_2w_3w_4w_5$ such that:
\begin{itemize}
\item
$|w_2w_3w_4| \leq p$;
\item
$w_2w_4 \neq 1$;
\item
$w_1w_2^nw_3w_4^nw_5 \in K$ for every $n \geq 0$.
\end{itemize}
For every $m \ge 0$, let
$$\begin{array}{ll}
\Lambda_m = \{&(u,v) \in \Fix(\varphi) \times \Fix(\varphi) \mid u \neq v \mbox{
  and } \exists q \in T(u)\\
&\exists w \in M_A: T(v) = T(u) \cup \oo{qT(w)},\;
|q| \geq m,\; ||w|| \leq p \}.
\end{array}$$
We show that
\begin{equation}
\label{noncf1}
\mbox{only finitely many }\Lambda_m \mbox{ are nonempty}.
\end{equation}
Indeed, suppose that (\ref{noncf1}) fails. This amounts to say that
$\Lambda_m \neq \emptyset$ for every $m \ge 0$. Fix $(u_m,v_m) \in
\Lambda_m$ and let $q_m,w_m$ be as in the definition of
$\Lambda_m$. Let
$$M = \max\{ ||a\varphi|| : a \in \wt{A} \}.$$
We claim that
\begin{equation}
\label{noncf2}
d(q_m,q_m\oo{\varphi}) < 2^{pM-m}.
\end{equation}
Indeed, we have $T(v_m) = T(u_m) \cup \oo{q_mT(w_m)}$. Hence
$$v_mv_m\inv = u_mu_m\inv q_mw_mw_m\inv q_m\inv$$ and so
\begin{equation}
\label{noncf3}
\begin{array}{lll}
v_mv_m\inv&=&(v_mv_m\inv)\varphi = (u_mu_m\inv q_mw_mw_m\inv
q_m\inv)\varphi\\
&=&u_mu_m\inv(q_mw_mw_m\inv
q_m\inv)\varphi.
\end{array}
\end{equation}
Since $q_mq_m\inv \geq u_mu_m\inv$, we have $(q_mq_m\inv)\varphi \geq
(u_mu_m\inv)\varphi = u_mu_m\inv$ and so $T(q_m\varphi) \subseteq
T(u_m)$. Hence (\ref{noncf3}) yields
$$T(v_m) = T(u_m) \cup
\oo{(q_m\varphi)T(w_m\varphi)} = T(u_m) \cup
\oo{(q_m\oo{\varphi})T(w_m\varphi)}.$$
Now it is easy to see that
$||w_m|| \leq p$ yields $||w_m\varphi|| \leq pM$. Let $a \in \wt{A}$
be such that $q_ma \in T(v_m) \setminus T(u_m)$. Then $q_ma =
\oo{(q_m\oo{\varphi})z}$ for some reduced word $z$ of length $\leq pM$
and so $q_m\oo{\varphi} = \oo{q_maz\inv}$. Since $|q_m| \geq m$ and
$z\inv$ can cancel at most $pM$ letters from the reduced word $q_ma$,
we get $r(q_m, q_m\oo{\varphi}) > m-pM$. Therefore (\ref{noncf2}) holds.

Now, since the completion $\wh{F_A}$ is compact, the sequence
$(q_m)_m$ must admit a convergent subsequence $(q_{i_m})_m$. Let
$\alpha = \lim_{m\to \infty} q_{i_m}$. We claim that $\alpha \in
\Fix(\wh{\varphi})$. Indeed, by continuity we have $\alpha\wh{\varphi}
= \lim_{m\to \infty} q_{i_m}\oo{\varphi}$. Let $\varepsilon >
0$. Since $\alpha = \lim_{m\to \infty} q_{i_m}$, there exists some $t
\geq 1$ such that
$$m \ge t \; \Rightarrow \; d(q_{i_m}, \alpha) < \varepsilon.$$
Moreover, we may assume that $2^{pM-t} < \varepsilon$. Thus, if $m \ge
t$, and since $d$ is an ultrametric, we get
$$\begin{array}{lll}
d(q_{i_m}\oo{\varphi}, \alpha)&\le&\max\{ d(q_{i_m}\oo{\varphi},
q_{i_m}), d(q_{i_m}, \alpha) \} < \max \{ 2^{pM-m}, \varepsilon \}\\
&\le&\max \{ 2^{pM-t}, \varepsilon \} < \varepsilon.
\end{array}$$
Hence $\alpha\wh{\varphi} = \alpha$ and so $\alpha \in
\Fix(\wh{\varphi})$. Since $|q_m| \geq m$ for every $m$, we have
$\alpha \neq 1$, a contradiction. Therefore (\ref{noncf1}) holds.

Suppose now that $\Fix(\varphi)$ is infinite. Since
$\Fix(\oo{\varphi}) = 1$, it follows from Lemma \ref{permv} that
$\Fix(\varphi) \subseteq E(M_A)$. Let $m > p$. Since $\Fix(\varphi)$
is infinite, there exists some $w \in K$ with $||w|| \ge
m$. We may assume that $w$ has minimal length among all such words.

Since $|w| \ge ||w|| \ge m > p$, there exists a
factorization $w = w_1w_2w_3w_4w_5$ such that:
\begin{itemize}
\item
$|w_2w_3w_4| \leq p$;
\item
$w_2w_4 \neq 1$;
\item
$w_1w_2^nw_3w_4^nw_5 \in K$ for every $n \geq 0$.
\end{itemize}
Let $u = w_1w_3w_5 \in K$. By minimality of $w$, we have $||u|| < m$,
hence $u \neq uw$ in $M_A$. We claim that
$$(u,uw) \in \Lambda_{m-p}.$$
First of all, we note that $\oo{u} = 1 = \oo{w}$ yields
$\oo{w_2w_3w_4} = \oo{w_3}$. There exists a path
$$1 \mapright{w_1} \oo{w_1} \mapright{w_3} \oo{w_1w_3} \mapright{w_5} 1$$
in $\MT(u)$. We have
$$T(u) = T(w_1) \cup \oo{w_1T(w_3)} \cup \oo{w_1w_3T(w_5)},$$
$$T(uw) = T(u) \cup T(w_1) \cup \oo{w_1T(w_2w_3w_4)} \cup
\oo{w_1w_2w_3w_4T(w_5)}.$$
Since $\oo{w_2w_3w_4} = \oo{w_3}$, we get
\begin{equation}
\label{noncf5}
T(uw) = T(u) \cup
\oo{w_1T(w_2w_3w_4)}.
\end{equation}
Note that $||w_2w_3w_4|| \leq |w_2w_3w_4| \leq p$. Suppose that
$|\oo{w_1}| < m-p$. Then the maximal length of a word in
$\oo{w_1T(w_2w_3w_4)}$ is $< m$, and in view of $||u|| < m$ and
(\ref{noncf5}) we get $||w|| < m$, a contradiction. Hence $|\oo{w_1}|
\geq m-p$. Together with (\ref{noncf5}), this yields $(u,uw) \in
\Lambda_{m-p}$ as claimed. Since $m$ is arbitrary large, this
contradicts (\ref{noncf1}). Therefore $\Fix(\varphi)$
is not context-free.

(iv) $\Rightarrow$ (iii) $\Rightarrow$ (ii) $\Rightarrow$ (i). Trivial.

(iv) $\Leftrightarrow$ (v) $\Leftrightarrow$ (vi). By Corollary \ref{ppff}.
\end{proof}

\begin{cor}
\label{grper}
Let $\varphi \in \End(M_A)$ be such that $\varphi|_{\wt{A}}$ is a
permutation without fixed points. Then $\Fix(\varphi)$ is not context-free.
\end{cor}

\begin{proof}
Clearly, $\oo{\varphi}$ is an automorphism and $\Fix(\wh{\varphi}) =
1$. Moreover, $\Per(\varphi) = M_A$. By Theorem \ref{noncf},
$\Fix(\varphi)$ is not context-free.
\end{proof}

The above corollary provides an infinite class of examples, for
arbitrary curl $> 1$, where $\Fix(\varphi)$ is not context-free.

\section{Radicals}\label{sec:radical}

We introduce now the concept of \emph{radical}, inspired by analogous
definitions in other contexts. Radicals occupy an
intermediate position between the submonoids of fixed points and
periodic points.

Given $\varphi\in\End(M_A)$ and $n \geq 1$, we define
$$
\Rad_{n}(\varphi) = \{ u \in \Fix(\varphi^n) \mid \oo{u} \in
\Fix(\oo{\varphi}) \}.$$
The following result summarizes some of the basic properties of
radicals:

\begin{lemma}
\label{radic}
Let $\varphi\in\End(M_A)$ and $m,n \geq 1$. Then:
\begin{itemize}
\item[(i)] $\Rad_{n}(\varphi)$ is an inverse submonoid of $M_A$;
\item[(ii)] $\Fix(\varphi) = \Rad_{1}(\varphi) \le \Rad_{n}(\varphi)
  \le \Per(\varphi)$;
\item[(iii)] if $m|n$, then $\Rad_{m}(\varphi) \le \Rad_{n}(\varphi)$;
\item[(iv)] $\Fix(\varphi) = \{ \mathcal{K}_{\varphi}(uu\inv)u \mid u
  \in \Rad_{n}(\varphi) \}$.
\end{itemize}
\end{lemma}

\begin{proof}
(i) -- (iii) Immediate.

(iv) Let $u \in \Rad_{n}(\varphi)$ and $v =
\mathcal{K}_{\varphi}(uu\inv)u$. Then $(uu\inv)\varphi^n = uu\inv$ and
it follows easily that $\mathcal{K}_{\varphi}(uu\inv) \in
\Fix(\varphi)$. Hence $vv\inv = \mathcal{K}_{\varphi}(uu\inv) \in
\Fix(\varphi)$. On the other hand, $\oo{v} = \oo{u} \in
\Fix(\oo{\varphi})$, and Lemma \ref{permv} yields $u \in
\Fix(\varphi)$.

Conversely, let $u \in \Fix(\varphi) \subseteq
\Rad_{n}(\varphi)$. Then $u = \mathcal{K}_{\varphi}(uu\inv)u$ and we
are done.
\end{proof}

Radicals may behave better than submonoids of fixed points:

\begin{theorem}
\label{cfrad}
Let $\varphi\in\End(M_A)$ and $N = \Curl(\oo{\varphi})$. Then
$\Rad_{N}(\varphi)$ is context-free.
\end{theorem}

\begin{proof}
By Lemma \ref{redcurl}, we have $\Curl(\oo{\varphi}^{N}) = 1$, hence
$\Fix(\varphi^N)$ is finitely generated by Theorem \ref{theo: product
  of phi-tiles}. Hence there exists a rational language $L \subseteq
\wt{A}^*$ such that $L\pi = \Fix(\varphi^N)$.

On the other hand, since $H = \Fix(\oo{\varphi})$ is finitely generated by
\cite{GT2}, its pre-image $H\sigma\inv$ is context-free by \cite{Sak}
(see also \cite[Ex. III.2.5]{Ber}). We claim that
$$\Rad_{N}(\varphi) = (L \cap H\sigma\inv)\pi.$$
Indeed, if $u \in \Rad_{N}(\varphi)$, then $u = v\pi$ for some $v \in
L$. Since $v\sigma = v\pi\sigma' = u\sigma' = \oo{u} \in H$, we get $v
\in L \cap H\sigma\inv$ and so $u \in (L \cap H\sigma\inv)\pi$. The
opposite inclusion is similar.

Since $L \cap H\sigma\inv$ is the intersection of a context-free
language with a rational language, it is context-free and so
$\Rad_{N}(\varphi)$ is context-free.
\end{proof}

\begin{theorem}
\label{nonrat}
Let $\varphi \in \End(M_A)$ be such that $\Fix(\oo{\varphi}) =
1$. Let $n \geq 1$. Then the following conditions are equivalent:
\begin{itemize}
\item[(i)]
$\Rad_{n}(\varphi)$ is rational;
\item[(ii)]
$\Rad_{n}(\varphi)$ is finite.
\end{itemize}
\end{theorem}

\begin{proof}
(i) $\Rightarrow$ (ii). Assume that $\Rad_{n}(\varphi) = L\pi$ for
some $L \subseteq \wt{A}^*$ rational. Then $L = L(\mathcal{A})$ for some finite
deterministic trim automaton $\mathcal{A} = (Q,q_0,T,E)$. Since
$\Fix(\oo{\varphi}) = 1$, we have $\oo{\Rad_{n}(\varphi)} = 1$
and so $\oo{L} = 1$.
Since $\mathcal{A}$ is trim, it follows that $\oo{z} = 1$
whenever $z$ labels a cycle in $\mathcal{A}$.

Suppose now that $\Rad_{n}(\varphi)$ is infinite and write $m =
|Q|$. There exists some $e \in \Rad_{n}(\varphi)$ with $||e|| \ge
m$. Take $u \in L \cap e\pi\inv$. Then $u$ must have some prefix $v$
such that $|\oo{v}| = ||e||$, hence there exists some path in
$\mathcal{A}$ of the form
$$q_0 \mapright{v} q \mapright{w} t \in T.$$
If we successively remove all nontrivial cycles from the path $q_0
\mapright{v} q$, we get a path $q_0 \mapright{v'} q$ of length $<
m$. However, since $\oo{z} = 1$
whenever $z$ labels a cycle in $\mathcal{A}$, we have $\oo{v'} =
\oo{v}$ and so $m \le ||e|| = |\oo{v}| = |\oo{v'}| \le |v'| < m$, a
contradiction. Therefore $\Rad_{n}(\varphi)$ is finite.

(ii) $\Rightarrow$ (i). Trivial.
\end{proof}

The next corollary provides an infinite class of examples, for
arbitrary curl $> 1$, where $\Rad_N(\varphi)$ is not rational:

\begin{cor}
\label{exnonrat}
Let $\varphi \in \End(M_A)$ be such that $\varphi|_{\wt{A}}$ is a
permutation without fixed points. Let $N = \Curl(\oo{\varphi})$. Then
$\Rad_{N}(\varphi)$ is context-free but not rational.
\end{cor}

\begin{proof}
By Theorem \ref{cfrad}, $\Rad_{N}(\varphi)$ is context-free. Since
$\oo{\varphi}|_{\wt{A}}$ is also a
permutation without fixed points, we have $\Per(\oo{\varphi}) = F_A$, hence
$\Fix(\oo{\varphi}^N) = F_A$ and so  $\oo{\varphi}^N = 1_{F_A}$. It
follows that
$\varphi^N = 1_{M_A}$. Since $\Fix(\oo{\varphi}) = 1$, we get
$\Rad_{N}(\varphi) = E(M_A)$ and so
$\Rad_{N}(\varphi)$ is infinite. Since $\Fix(\oo{\varphi}) = 1$, we
may apply Theorem \ref{nonrat}, thus $\Rad_{N}(\varphi)$ is not rational.
\end{proof}

Finally, we use Lemma \ref{radic}(iv) to show that $\Fix(\varphi)$ is
always context-sensitive:

\begin{theorem}
\label{consen}
Let $\varphi \in \End(M_A)$. Then $\Fix(\varphi)$ is
context-sensitive.
\end{theorem}

\begin{proof}
Context-sensitive languages can be characterized as languages accepted
by {\em linear-bounded Turing machines}, i.e. Turing machines
$\mathcal{T}$ for which there exists some constant $K \ge 1$ such that
any word $u \in L(\mathcal{T})$ can be accepted by some computation
using only the first $K|u|$ cells of the tape.

By Theorem \ref{perio}, we have $\Per(\varphi) = \Fix(\varphi^m)$ for
some $m \ge 1$. Let $A_1, \ldots,A_{m-1}$ be disjoint
copies of $A$, inducing bijections $\alpha_i:\wt{A} \to \wt{A_i}$
for $i = 1, \ldots, m-1$. Let
$$B_i = \{ a \in A \mid a\varphi^i \neq 1\}$$
for $i = 1, \ldots, m-1$. We define a
matched homomorphism $\beta_i: \wt{A}^* \to \wt{B_i\alpha_i}^*$ by
$$a\beta_i = \left\{
\begin{array}{ll}
a\alpha_i&\mbox{ if }a \in B_i\\
1&\mbox{ if }a \in A\setminus B_i
\end{array}
\right.$$
Let $N = \Curl(\oo{\varphi})$.
By Theorem \ref{cfrad}, $\Rad_{N}(\varphi)$ is context-free. Hence we
have $\Rad_{N}(\varphi)
= C\pi$ for some context-free language $C \subseteq \wt{A}^*$.
We define
$$L = \{ uu\inv(u\beta_1)(u\inv\beta_1)(u\beta_2)(u\inv\beta_2)\ldots
(u\beta_{m-1})(u\inv\beta_{m-1})u \mid u \in C \}.$$
It is easy to see that $L$ is accepted by a linear-bounded Turing
machine $\mathcal{T}$. Indeed, since $C$ is context-free, $C$ is
itself accepted by
a linear-bounded Turing machine $\mathcal{T}'$. We can make
$\mathcal{T}$ consider all the possible factorizations of the input
potentially leading to some word of the form
$$uu\inv(u\beta_1)(u\inv\beta_1)\ldots
(u\beta_{m-1})(u\inv\beta_{m-1})u$$
(using multiple tracks to help),
then using $\mathcal{T}'$ as a subroutine and checking each one of the
presumed factors using the definitions of the $B_i$ and
$\beta_i$. Therefore $L$ is context-sensitive.

For $i = 1, \ldots, m-1$, fix a matched endomorphism $\psi_i$ of
$\wt{A}^*$ satisfying $\psi_i\pi = \pi\varphi^i$ (it suffices to have
$a\psi_i\pi = a\pi\varphi^i$ for every $a \in A$, so it really exists).
Let $X = A \cup B_1\alpha_1 \cup \ldots \cup B_{m-1}\alpha_{m-1}$ and
define a matched homomorphism $\gamma:
\wt{X}^* \to \wt{A}^*$ as follows. Given $a \in A$, let $a\gamma = a$.
Given $a \in B_i$, let $a\alpha_i\gamma = a\psi_i$. Note that
$a\alpha_i\gamma \neq 1$ since $a\varphi^i \neq 1$ by definition of $B_i$.

Since context-sensitive languages are closed under $\varepsilon$-free
homomorphisms, it follows that $L\gamma$ is context-sensitive.
Note that $a\alpha_i\gamma\pi = a\psi_i\pi = a\pi\varphi^i$ whenever
$a \in B_i$, so $a\beta_i\gamma\pi = a\pi\varphi^i$ for every
$a \in A$ (if $a \notin B_i$, then $a\beta_i = 1 =
a\pi\varphi^i$). Thus we get
$$\begin{array}{lll}
L\gamma\pi&=&\{ uu\inv(u\beta_1)(u\inv\beta_1)\ldots
(u\beta_{m-1})(u\inv\beta_{m-1})u \mid u \in C \}\gamma\pi\\
&=&\{ vv\inv(v\varphi)(v\inv\varphi)\ldots
(v\varphi^{m-1})(v\inv\varphi^{m-1})v \mid v \in C\pi \}.
\end{array}$$
Recall that $C\pi = \Rad_{N}(\varphi)$. Moreover, we claim that
\begin{equation}
\label{consen1}
vv\inv(v\varphi)(v\inv\varphi)\ldots
(v\varphi^{m-1})(v\inv\varphi^{m-1}) = \mathcal{K}_{\varphi}(vv\inv).
\end{equation}
Indeed, if $p$ is the period of $vv\inv$, then
$\mathcal{K}_{\varphi}(vv\inv) = vv\inv (vv\inv)\varphi \ldots
(vv\inv)\varphi^{p-1}$. Since $vv\inv \in \Per(\varphi) =
\Fix(\varphi^m)$, we have $p \le m$ and so (\ref{consen1})
holds. Therefore, by Lemma \ref{radic}(iv), we get
$$\Fix(\varphi) = \{ \mathcal{K}_{\varphi}(vv\inv)v
\mid v \in \Rad_{N}(\varphi) \} = L\gamma\pi.$$
Since $L\gamma$ is context-sensitive, so is $\Fix(\varphi)$.
\end{proof}

\section{Open Problems}\label{sec: open problems}
We give a list of natural open problems originated by the previous results.

\begin{prob}
\label{prob1}
To find examples of $\varphi \in \End(M_A)$ such that:
\begin{itemize}
\item[(i)] $\Fix(\varphi)$ is context-free but not rational;
\item[(ii)] $\Fix(\varphi)$ is rational but not finitely generated.
\end{itemize}
\end{prob}

\begin{prob}
\label{prob2}
To show that a basis of $\Per(\theta)$ can be computed for every
$\theta \in \End(F_A)$.
\end{prob}

\begin{prob}
\label{prob3}
To show that $\Curl(\theta)$ can be computed for every
$\theta \in \End(F_A)$.
\end{prob}

If Problem \ref{prob2} is solved, then Problem \ref{prob3} is solved
as well: we just compute basis of $\Per(\theta)$, $\Per(\theta^{2!})$,
$\Per(\theta^{3!})$ until we reach $\Per(\theta^{n!}) =
\Per(\theta^{(n+1)!})$ for some $n$ (and we know we eventually will)
by Proposition \ref{perfix}.

\begin{prob}
\label{prob4}
To compute $\St_{\varphi}$ for an arbitrary $\varphi \in \End(M_A)$.
\end{prob}

If we compute the $\varphi$-stable letters, we can compute the
$\varphi$-tiles in $\mathcal{T}_{\varphi}$.

\section*{Acknowledgments}

The authors acknowledge support from the European Regional Development
Fund through the
programme COMPETE and by the Portuguese Government through the FCT --
Funda\c c\~ao para a Ci\^encia e a Tecnologia under the project
PEst-C/MAT/UI0144/2011.
The first author also acknowledges the support
of the FCT project SFRH/BPD/65428/2009.


\begin{thebibliography}{99}
\bibitem{Ber} J.~Berstel, {\em Transductions and Context-free
Languages}, Teubner, 1979.
\bibitem{CS} J.~Cassaigne and P.~V.~Silva, Infinite words and
  confluent rewriting systems: endomorphism
extensions, {\em Internat. J. Algebra Comput.} 19(4) (2009),
443--490.
\bibitem{CS2} J.~Cassaigne and P.~V.~Silva, Infinite periodic points of
  endomorphisms over special confluent rewriting systems, {\em
    Ann. Inst. Fourier} 59(2) (2009), 769--810.
\bibitem{Dug} J. Dugundji, {\em Topology}, Allyn and Bacon, 1966.
\bibitem{GT2} R.~Z.~Goldstein and E.~C.~Turner, Fixed subgroups of
  homomorphisms of free groups, {\em Bull. London Math. Soc.} 18
  (1986), 468--470.
\bibitem{HMU} J.~E.~Hopcroft, R.~Motwani and J.- D.~Ullman,
  {\em Introduction to Automata Theory, Languages and Computation},
  3rd ed., Addison-Wesley 2006.
\bibitem{IT} W.~Imrich and E.~C.~Turner, Endomorphisms of free groups
  and their fixed points, {\em Math. Proc. Cambridge Philos. Soc.} 105
  (1989), 421--422.
\bibitem{LS} R.~C.~Lyndon and P.~E.~Schupp, {\em Combinatorial Group
Theory}, Springer-Verlag 1977.
\bibitem{FreeInvGraphImm} S.~W.~Margolis and J.~C.~Meakin, Free
  inverse monoids and graph immersions, {\em Internat. J. Algebra
    Comput.} 3(1) (1993), 79--99.
\bibitem{Mun} W.~D.~Munn, Free inverse semigroups, {\em Proc. London
    Math. Soc.} 30 (1974), 385--404.
\bibitem{PP} D.~Perrin and J.-E.~Pin, {\em Infinite Words: Automata,
Semigroups, Logic and games}, Pure and Applied Mathematics Series
141, Elsevier Academic Press, Amsterdam, 2004.
\bibitem{Pet} M.~Petrich, {\em Inverse Semigroups}, Wiley, New York,
  1984.
\bibitem{RSS} E.~Rodaro, P.~V.~Silva and M.~Sykiotis, Fixed points of
endomorphisms of graph groups, {\em J. Group Theory} (to appear).
\bibitem{RS} E.~Rodaro and P.~V.~Silva Fixed points of
endomorphisms of endomorphisms of trace monoids, arXiv:1211.4517
\bibitem{Sak} J.~Sakarovitch, Sur les groupes infinis, consid\'er\'es
  comme mono\"ıdes syntaxiques de langages formels, In: {\em S\'eminaire
  d’Alg\`ebre Paul Dubreil, 29\`eme ann\'ee (Paris, 1975–1976)}, pages
  168--179, Lecture Notes in Math. 586, Springer, Berlin, 1977.
\bibitem{Sch} H.~E.~Scheiblich, Free inverse semigroups, {\em
    Proc. Amer. Math. Soc.} 38 (1973), 1--7.
\bibitem{Sil1} P.~V.~Silva, Fixed points of endomorphisms over special
  confluent rewriting
  systems, {\em Monatsh. Math.} 161(4) (2010), 417--447.
\bibitem{Sil2} P.~V.~Silva, Fixed points of endomorphisms of certain
 free products, {\em  RAIRO Theoret. Inf. Appl.} 46 (2012), 165--179.
\bibitem{Sil3} P.~V.~Silva, Fixed points of endomorphisms of virtually
  free groups, {\em Pacific J. Math.} (to appear).
\bibitem{Tak} M.~Takahasi, Note on chain conditions in free groups,
  {\em Osaka Math. J.} 3 (1951) 221--225.
\end{thebibliography}
\end{document}